\documentclass[12pt]{amsart}
\usepackage{amssymb}
\usepackage{xcolor}
\usepackage{mathrsfs}
\usepackage{amsfonts}
\usepackage{amsmath}
\numberwithin{equation}{section}
\allowdisplaybreaks[4]
\newtheorem{Theorem}{Theorem}[section]
\newtheorem{Proposition}[Theorem]{Proposition}

\newtheorem{Lemma}[Theorem]{Lemma}

\newtheorem{Definition}[Theorem]{Definition}

\unless\ifcsname proof\endcsname

\fi
\def\supp{{\rm supp}\ }
\def\bbZ{\mathbb{Z}}
\def\bbR{\mathbb{R}}

\begin{document}

\title[Multilinear Fractional Maximal Operators] { Characterization of a Two Weight Inequality for Multilinear Fractional Maximal Operators}
\thanks{This work
was supported partially by the National Natural Science Foundation of China (10990012)  and the Research Fund for the Doctoral Program
of Higher Education.}

\author{Kangwei Li}

\author{Wenchang Sun}
\thanks{Corresponding author: Wenchang Sun}
\email{likangwei9@mail.nankai.edu.cn, sunwch@nankai.edu.cn}

\address{Department of Mathematics and LPMC,  Nankai University,
      Tianjin~300071, China}

 \begin{abstract}
In this paper, we study the characterization of two weight inequality for multilinear fractional maximal operators.
We give a multilinear analogue of Sawyer's two weight test condition.
\end{abstract}

\maketitle

\section{Introduction and Main Results}
We begin with the $A_p$ condition introduced by Muckenhoupt. For a weight $w$, i.e. a non-negative locally integrable function,
 we call $w$ satisfies the $A_p$ condition
if
\[
  [w]_{A_p}:=\sup_{Q: \mbox{cubes in $\bbR^n$}}\left(\frac{1}{|Q|}\int_Q w(x)dx\right)\left(\frac{1}{|Q|}\int_Q w(x)^{1-p'}dx\right)^{p-1}<\infty,
\]
where $p'$ is the dual exponent of $p$ defined by the equation $1/p+1/{p'}=1$.
In \cite{Mu}, Muckenhoupt showed that the Hardy-Littlewood maximal function
\[
  Mf(x):=\sup_{Q\ni x}\frac{1}{|Q|}\int_Q |f(y)|dy
\]
is bounded on $L^p(w)$  if and only if $w$ satisfies the $A_p$ condition.

In \cite{MW}, Muckenhoupt and Wheeden characterized the weighted strong-type inequality for fractional operators
in terms of the so-called $A_{p,q}$ condition. For $0<\alpha<n$, $1<p<n/\alpha$ and $1/q=1/p-\alpha/n$, they showed that the
fractional maximal function
\[
   M_\alpha f(x):=\sup_{Q\ni x}\frac{1}{|Q|^{1-\alpha/n}}\int_Q |f(y)|dy
\]
is bounded from $L^p(w^p)$ to $L^{p}(w^q)$ if and only if $w$ satisfies the following  $A_{p,q}$ condition,
\[
  [w]_{A_{p,q}}:=\sup_{Q: \mbox{cubes in $\bbR^n$}}\left(\frac{1}{|Q|}\int_Q w(x)^qdx\right)\left(\frac{1}{|Q|}\int_Q w(x)^{-p'}dx\right)^{q/{p'}}<\infty.
\]

In \cite{Buckley}, Buckley showed that for $1<p<\infty$, $\|M\|_{L^p(w)\rightarrow L^p(w)}\le c[w]_{A_p}^{p'/p}$ and the exponent $p'/p$ is the best possible.
%For the Calder\'on-Zygmund operator $T$, Hyt\"onen \cite{H}
%obtained the sharp result is $\|T\|_{L^p(w)\rightarrow L^p(w)}\le c[w]_{A_p}^{\max\{1, %p'/p\}}$.
In \cite{LMPT}, Lacey, Moen, P\'erez and Torres obtained the analogous result for $M_\alpha$.
They proved the sharp inequality
$\|M_\alpha\|_{L^p(w^p)\rightarrow L^q(w^q)}\le c[w]_{A_{p,q}}^{\frac{p'}{q}(1-\alpha/n)}$.

For the two weight case, we can consider similar problem. That is,
to find a condition for a pair of weights $(u,v)$ such that $M_\alpha$ is bounded from
$L^p(u)$ to $L^q(v)$ for all $0\le\alpha<n$. In \cite{S}, Sawyer gave a characterization of a two weight inequality, showing that $M_\alpha$ is bounded
from $L^p(u)$ to $L^q(v)$ if and only if $(u,v)$ satisfies that
\[
  [u,v]_{S_{p,q}}:=\sup_{Q: \mbox{cubes in $\bbR^n$}}\frac{\left(\int_{Q}M_\alpha(\sigma 1_Q)^q vdx\right)^{1/q}}{\sigma(Q)^{1/{p}}}<\infty,
\]
where $\sigma=u^{1-p'}$, $0\le\alpha<n$, $1<p<n/\alpha$ and $1/q=1/p-\alpha/n$.
The above inequality is known as Sawyer's test condition.
In \cite{M1}, Moen improved Sawyer's result by showing that
\[
  \|M_\alpha\|_{L^p(u)\rightarrow L^q(v)}\asymp [u,v]_{S_{p,q}}.
\]
We refer the readers to \cite{LSSU1, LSU, NTV, S1, SSU} for more backgrounds and the new
breakthrough of the two weight characterization
 of singular integrals using Sawyer's test condition.

Now we move on the story to the multilinear case. We study the multilinear fractional maximal operator.
 For $0\le\alpha< mn$, the multilinear fractional maximal function $\mathcal{M}_\alpha$ is defined by
\[
   \mathcal{M}_\alpha (\vec f)(x)=\sup_{Q\ni x}\prod_{i=1}^m\frac{1}{|Q|^{1-\alpha/{mn}}}\int_Q |f_i(y_i)|dy_i.
\]
Specially, when $\alpha=0$, $\mathcal{M}_0$ is the multilinear maximal function denoted by $\mathcal{M}$ which is defined by
\[
  \mathcal{M}(f_1,\cdots, f_m)(x)=\sup_{Q\ni x}\prod_{i=1}^m\frac{1}{|Q|}\int_Q |f_i(y_i)|dy_i.
\]
The dyadic multilinear fractional maximal function is defined by
\[
  \mathcal{M}^{\mathscr{D}}_\alpha(f_1,\cdots, f_m)(x)=\sup_{Q\ni x,Q\in\mathscr{D}}\prod_{i=1}^m\frac{1}{|Q|^{1-\alpha/{mn}}}\int_Q |f_i(y_i)|dy_i,
\]
where $\mathscr{D}$ is a dyadic grid in $\bbR^n$, for which the
definition is given in the next section.

In \cite{LOPTT}, Lerner, Ombrosi, P\'erez, Torres and Trujillo-Gonz\'alez introduced the multiple $A_{\vec{P}}$ weights.
Let $\vec{P}=(p_1,\cdots,p_m)$ with $1\le p_1,\cdots,p_m<\infty$ and $1/{p_1}+\cdots+1/{p_m}=1/p$. Given $\vec{w}=(w_1,\cdots, w_m)$, set
\[
  v_{\vec{w}}=\prod_{i=1}^m w_i^{p/{p_i}}.
\]
We say that $\vec{w}$ satisfies the multilinear $A_{\vec{P}}$ condition if
\[
  [\vec{w}]_{A_{\vec{P}}}:=\sup_Q \left(\frac{1}{|Q|}\int_Q v_{\vec{w}}\right)\prod_{i=1}^m\left( \frac{1}{|Q|}\int_Q w_i^{1-p_i'}\right)^{p/{p_i'}}<\infty,
\]
when $p_i=1$, $( \frac{1}{|Q|}\int_Q w_i^{1-p_i'})^{1/{p_i'}}$ is
understood as $(\inf_Q w_i)^{-1}$. They proved that $\mathcal{M}$ is bounded from $L^{p_1}(w_1)\times\cdots\times L^{p_m}(w_m)$ to $L^p(v_{\vec{w}})$ if
and only if $\vec{w}\in A_{\vec{P}}$.

In \cite{M}, Moen introduced the multiple $A_{\vec{P},q}$ weight.
Let $1/p_1 + \cdots + 1/p_m = 1/q + \alpha/n$. A multiple weight $(w_1,\cdots, w_m)$ is said
to belong to the $A_{\vec{P},q}$ class if and only if
\[
 [\vec{w}]_{A_{\vec{P},q}}:= \sup_{Q}\left(\frac{1}{|Q|}\int_Q \Pi_{i=1}^m w_i^q
    \right)\prod_{i=1}^m \left(\frac{1}{|Q|}\int_Q w_i^{-p_i'}\right)^{q/{p_i'}}<\infty,
\]
Moen showed that
$\mathcal{M}_\alpha$ is bounded from $L^{p_1}(w_1^{p_1})\times\cdots\times L^{p_m}(w_m^{p_m})$
to $L^q(u_{\vec{w}})$ if and only if $ \vec{w}\in A_{\vec{P},q}$.

For the two weight case,
recently, Chen and Dami\'an \cite{CD} gave some sufficient conditions for
 the two weight inequality to hold for multilinear maximal operators.
In this paper, we prove the following result.
\begin{Theorem}\label{thm:main}
Suppose that $0\le \alpha<mn$, that $1<p_1,\cdots,p_m<\infty$, that $1/{p}=1/{p_1}+\cdots+1/{p_m}$, that $1/q=1/p-\alpha/n$ and that $q\ge \max_i\{p_i\}$.
Let $w_1$, $\cdots$, $w_m$, $v$ be weights and set $\sigma_i=w_i^{1-p_i'}$, $i=1,\cdots,m$. Define
\[
  [\vec{w},v]_{S_{\vec{P},q}}:=\sup_{Q: \mbox{cubes in $\bbR^n$}}\frac{\left(\int_{Q}\mathcal{M}_\alpha(\sigma_1 1_Q,\cdots,\sigma_m 1_Q)^q vdx\right)^{1/q}}{\prod_{i=1}^m \sigma_i(Q)^{1/{p_i}}}.
\]
Then $\mathcal{M}_\alpha$ is bounded from $L^{p_1}_{w_1}(\bbR^n)\times\cdots\times  L^{p_m}_{w_m}(\bbR^n)$ to $L^q_v(\bbR^n)$
if and only if $[\vec{w},v]_{S_{\vec{P},q}}$ is finite. Moreover,
\[
  \|\mathcal{M}_\alpha\|_{L^{p_1}(w_1)\times\cdots\times  L^{p_m}(w_m)\rightarrow L^q(v)}\asymp[\vec{w},v]_{S_{\vec{P},q}}.
\]
\end{Theorem}

In the rest of this paper, we give a proof of Theorem~\ref{thm:main}.
And in Section 4, we give other test conditions which cover all indices.

\section{Preliminary Results}
By a general dyadic grid $\mathscr{D}$ we mean a collection of cubes with the following
properties: (i) for any $Q\in\mathscr{D}$ its sidelength $l_Q$ is of the form $2^k$, $k\in\bbZ$; (ii)
 $Q\cap R \in \{Q,R,\emptyset\}$ for any $Q,R\in\mathscr{D}$; (iii) the cubes of a fixed sidelength $2^k$ form a partition
of $\bbR^n$.

We say that  $\mathcal{S}:=\{Q_{j,k}\}$ is a sparse family of cubes if:
\begin{enumerate}
\item for each fixed $k$ the cubes $Q_{j,k}$ are pairwise disjoint;
\item if $\Gamma_k=\bigcup_j Q_{j,k}$, then $\Gamma_{k+1}\subset \Gamma_k$;
\item $|\Gamma_{k+1}\bigcap Q_{j,k}|\le \frac{1}{2}|Q_{j,k}|$.
\end{enumerate}
For any $Q_{j,k}\in\mathcal{S}$, we define $E(Q_{j,k})=Q_{j,k}\setminus \Gamma_{k+1}$.
Then the sets $E(Q_{j,k})$ are pairwise disjoint and $|E(Q_{j,k})|\ge \frac{1}{2}|Q_{j,k}|$.

Define
\[
  \mathscr{D}_t:=\{2^{-k}([0,1)^n+m+(-1)^k t): k\in\bbZ, m\in\bbZ^n\},\quad t\in\{0, 1/3\}^n.
\]
The importance of these grids is shown by the following proposition, which
 can be found in \cite[proof of Theorem 1.10]{HP}, see also \cite[Proposition 5.1]{L}.
\begin{Proposition}\label{prop:p1}
There are $2^n$ dyadic grids $\mathscr{D}_t$, $t\in\{0,1/3\}^n$ such that for any cube
$Q\subset \bbR^n$ there exists a cube $Q_t \in \mathscr{D}_t$ satisfying
 $Q \subset Q_t$ and $l(Q_t)\le 6l(Q)$.
\end{Proposition}

For any weight $\sigma$, cube $Q$ and locally integrable function $f$ with respect to the measure $\sigma dx$, define the average $\mathbb{E}_{Q}^\sigma f:=\sigma(Q)^{-1}\int_Q f\sigma$.
First we give the definition of the principal cubes, which is introduced in \cite[Definition 8.2]{HLMORSU}.
\begin{Definition}[Principal cubes]\label{def:d1}
We form the collection $\mathcal{G}$ of principal cubes as follows.
Let $\mathcal{G}_0:=\{ \overline{Q}\}$ (the maximal dyadic cube that we consider). And inductively,
\[
  \mathcal{G}_k:=\bigcup_{G\in \mathcal{G}_{k-1}}\{G'\subset G: \mathbb{E}_{G'}^\sigma |f|>4 \mathbb{E}_{G}^\sigma |f|, G' \mbox{is a maximal such dyadic cube}\}.
\]
Let $\mathcal{G}:=\bigcup_{k=0}^\infty \mathcal{G}_k$.
For any dyadic $Q(\subset \overline{Q})$, we let
\[
  \Gamma(Q):=\mbox{the minimal principal cube containing $Q$}.
\]
\end{Definition}
It follows from the definition that
\[
  \mathbb{E}_{Q}^\sigma |f|\le 4\mathbb{E}_{\Gamma(Q)}^\sigma |f|.
\]

By the definition, if $f\in L^p(\sigma)$, we immediately have
\begin{equation}\label{eq:e2}
\sum_{G\in\mathcal{G}}(E_G^\sigma |f|)^p\sigma(G)\le C\|M_\sigma^{\mathscr{D}}f\|_{L^p(\sigma)}^p\le C_p\|f\|_{L^p(\sigma)}^p.
\end{equation}

\begin{Proposition}\cite[Theorem 1.3.1]{BL}\label{prop:inter}
Assume that $p_0\neq p_1$ and that
\begin{eqnarray*}
&& T: L^{p_0}(U, d\mu)\rightarrow L^{q_0, \infty}(V, d\nu) \quad\mbox{with norm $M_0$,}\\
&& T: L^{p_1}(U, d\mu)\rightarrow L^{q_1, \infty}(V, d\nu) \quad\mbox{with norm $M_1$.}
\end{eqnarray*}
Put
\[
  \frac 1 p=\frac{1-\theta}{p_0}+\frac \theta {p_1}, \quad\frac 1 q=\frac{1-\theta}{q_0}+\frac \theta {q_1},
\]
and assume that $p\le q$.
Then
\[
  T: L^{p}(U, d\mu)\rightarrow L^{q}(V, d\nu)
\]
with norm $M$ satisfying
\[
  M\le C_\theta M_0^{1-\theta}M_1^\theta.
\]
\end{Proposition}

\section{Proof of Theorem~\ref{thm:main}}
We only prove Theorem~\ref{thm:main} for $m=2$, since the general case can be proved similarly.
Firstly, we prove the following lemma.
\begin{Lemma}\label{lm:l}
Suppose that $0\le \alpha<2n$, that $1<p_1,p_2<\infty$, that $1/{p}=1/{p_1}+1/{p_2}$, that $1/q=1/p-\alpha/n$ and that $q\ge p_2$.
Let $(w_1,w_2, v)$ be weights and set $\sigma_i=w_i^{1-p_i'}$.
Then for dyadic grid $\mathscr{D}$ and function $f$ with $\supp f\subset R \in\mathscr{D}$,
\[
  \|1_R\mathcal{M}_\alpha^{\mathscr{D}}(1_{R}\sigma_1, f\sigma_2)\|_{L^q(v)}\lesssim[\vec{w},v]_{S_{\vec{P},q}}\sigma_1(R)^{1/{p_1}}\|f\|_{L^{p_2}(\sigma_2)}.
\]
\end{Lemma}
\begin{proof}
Similarly to  \cite{LMS}, let $a=2^{(2-\alpha/n)(n+1)}$ and
\[
  \Omega_k=\{x\in\bbR^n: \mathcal{M}_\alpha^{\mathscr{D}}(1_{R}\sigma_1, f\sigma_2)(x)>a^k\}.
\]
Then we have $\Omega_k=\cup_j Q_j^k$, where $Q_j^k$ are disjoint maximal dyadic cubes in $\Omega_k$
and $\{Q_j^k\}$ is a sparse family in $\mathscr{D}$ and
\[
  a^k<\frac{1}{|Q_j^k|^{2-\alpha/n}}\int_{Q_j^k}1_{R}\sigma_1\int_{Q_j^k}|f|\sigma_2\le 2^{2n-\alpha}a^k.
\]
It follows that
\begin{eqnarray*}
\mathcal{M}_\alpha^{\mathscr{D}}(1_{R}\sigma_1, f\sigma_2)&\asymp&
\sum_{k,j}\left(\frac{1}{|Q_j^k|^{2-\alpha/n}}\int_{Q_j^k}1_{R}\sigma_1\int_{Q_j^k}|f|\sigma_2\right) 1_{E(Q_j^k)}.
\end{eqnarray*}
Therefore,
\begin{eqnarray*}
&&1_R\mathcal{M}_\alpha^{\mathscr{D}}(1_{R}\sigma_1, f\sigma_2)\\&\asymp&
\sum_{k,j: Q_j^k\subset R}\left(\frac{1}{|Q_j^k|^{2-\alpha/n}}\int_{Q_j^k}1_{R}\sigma_1\int_{Q_j^k}|f|\sigma_2\right) 1_{E(Q_j^k)}\\
&&\quad+\sum_{k,j: Q_j^k\supsetneq R}\left(\frac{1}{|Q_j^k|^{2-\alpha/n}}\int_{Q_j^k}1_{R}\sigma_1\int_{Q_j^k}|f|\sigma_2\right)1_{E(Q_j^k)\cap R}\\
&\lesssim&\sum_{k,j: Q_j^k\subset R}\left(\frac{1}{|Q_j^k|^{2-\alpha/n}}\int_{Q_j^k}1_{R}\sigma_1\int_{Q_j^k}|f|\sigma_2\right) 1_{E(Q_j^k)}\\
&&\quad+\left(\frac{1}{|R|^{2-\alpha/n}}\sigma_1(R)\int_{R}|f|\sigma_2\right) 1_{R}.
\end{eqnarray*}
Consequently, we have
\begin{eqnarray*}
\int_{R}\mathcal{M}_\alpha^{\mathscr{D}}(1_{R}\sigma_1, f\sigma_2)^q v dx
&\lesssim&\sum_{k,j: Q_j^k \subset R}\left(\frac{\int_{Q_j^k}1_{R}\sigma_1\int_{Q_j^k}|f|\sigma_2}{|Q_j^k|^{2-\alpha/n}} \right)^q v(E(Q_j^k))\\
&&\quad
 +\int_{R}\mathcal{M}_\alpha^{\mathscr{D}}(1_{R}\sigma_1, 1_R\sigma_2)^q v dx \bigg(\frac{1}{\sigma_2(R)}\int_R |f_2|\sigma_2\bigg)^q\\
&\le&\sum_{k,j: Q_j^k \subset R}\left(\frac{\int_{Q_j^k}1_{R}\sigma_1\int_{Q_j^k}|f|\sigma_2}{|Q_j^k|^{2-\alpha/n}} \right)^q v(E(Q_j^k))\\
&&\quad+[\vec{w},v]_{S_{\vec{P},q}}^q\sigma_1(R_1)^{q/{p_1}}\|f\|_{L^{p_2}(\sigma_2)}^q.
\end{eqnarray*}
Now we reduce the problem to estimate the following, followed by Sawyer's technique \cite{S}, we have
\begin{eqnarray*}
&&\sum_{k,j: Q_j^k \subset R}\left(\frac{\int_{Q_j^k}1_{R}\sigma_1\int_{Q_j^k}|f|\sigma_2}{|Q_j^k|^{2-\alpha/n}} \right)^q v(E(Q_j^k))\\
&=&\sum_{k,j: Q_j^k \subset R}(E_{Q_j^k}^{\sigma_2}|f|)^q\left(\frac{\sigma_1(Q_j^k)\sigma_2(Q_j^k)}{|Q_j^k|^{2-\alpha/n}} \right)^q v(E(Q_j^k))\\
&=&\sum_{k,j: Q_j^k \subset R}(E_{Q_j^k}^{\sigma_2}|f|)^q \gamma_j^k,
\end{eqnarray*}
where
\[
  \gamma_j^k:=\left(\frac{\sigma_1(Q_j^k)\sigma_2(Q_j^k)}{|Q_j^k|^{2-\alpha/n}} \right)^q v(E(Q_j^k)).
\]
Now let $\Omega_R:=\{(k,j): Q_j^k\subset R\}$ and let $\gamma$ be the measure on $\Omega_R$ that assigns mass $\gamma_j^k$ to $(k,j)$.
Define
\[
  T: (L^1+L^\infty)(\bbR^n, \sigma_2dx)\rightarrow L^\infty(\Omega_R, d\gamma)
\]
by
\[
  Tf:=\{E_{Q_j^k}^{\sigma_2}|f|\}_{(k,j)\in\Omega_R}, \quad f\in (L^1+L^\infty)(\bbR^n, \sigma_2dx).
\]
Clearly $T$ is sublinear and of strong-type $(\infty,\infty)$ with norm $1$. Next we show that $T$ is of weak-type $(1,q/{p_2})$.
Let $\lambda>0$ and $\{I_i\}_i$ denotes the maximal cubes relative to the collection
\[
  \{Q_j^k: E_{Q_j^k}^{\sigma_2}|f|>\lambda, Q_j^k\subset R\}
\]
We have
\begin{eqnarray*}
\gamma\{Tf>\lambda\}
&=& \sum_{\{(j,k): E_{Q_j^k}^{\sigma_2}|f|>\lambda\}} \gamma_j^k
\le \sum_i \sum_{Q_j^k\subset I_i}\gamma_j^k\\
&\le&\sum_i \sum_{Q_j^k\subset I_i}\int_{E(Q_j^k)}\mathcal{M}_\alpha^{\mathscr{D}}(1_{I_i}\sigma_1, 1_{I_i}\sigma_2)^q v dx\\
&\le&\sum_i \int_{I_i}\mathcal{M}_\alpha^{\mathscr{D}}(1_{I_i}\sigma_1, 1_{I_i}\sigma_2)^q v dx\\
&\le&[\vec{w},v]_{S_{\vec{P},q}}^q\sum_i \sigma_1(I_i)^{q/{p_1}}\sigma_2(I_i)^{q/{p_2}}\\
&\le&[\vec{w},v]_{S_{\vec{P},q}}^q\left(\sum_i \sigma_1(I_i)^{p/{p_1}}\sigma_2(I_i)^{p/{p_2}}\right)^{q/p}\\
&\le&[\vec{w},v]_{S_{\vec{P},q}}^q\left(\sum_i \sigma_1(I_i)\right)^{q/{p_1}}\left(\sum_i \sigma_2(I_i)\right)^{q/{p_2}}\\
&\le&[\vec{w},v]_{S_{\vec{P},q}}^q\sigma_1(R)^{q/{p_1}}\left(\frac{1}{\lambda}\int |f|\sigma_2 dx\right)^{q/{p_2}}.
\end{eqnarray*}
This shows that $T$ is of weak-type $(1,q/{p_2})$ with norm $[\vec{w},v]_{S_{\vec{P},q}}^{p_2}\sigma_1(R)^{{p_2}/{p_1}}$.
Then by Proposition~\ref{prop:inter} we get $T$ is of strong-type $(p_2, q)$ with norm $C[\vec{w},v]_{S_{\vec{P},q}}\sigma_1(R)^{1/{p_1}}$,
which is exactly the following
\[
  \sum_{k,j: Q_j^k \subset R}\left(\frac{\int_{Q_j^k}1_{R}\sigma_1\int_{Q_j^k}|f|\sigma_2}{|Q_j^k|^{2-\alpha/n}} \right)^q v(E(Q_j^k))
  \lesssim [\vec{w},v]_{S_{\vec{P},q}}^q\sigma_1(R)^{q/{p_1}}\|f\|_{L^{p_2}(\sigma_2)}^q.
\]
This completes the proof.
\end{proof}

Now we are ready to prove the following, which is very close to our main result.
\begin{Lemma}\label{lm:ll}
Suppose that $0\le \alpha<2n$, that $1<p_1,p_2<\infty$, that $1/{p}=1/{p_1}+1/{p_2}$, that $1/q=1/p-\alpha/n$ and that $q\ge \max\{p_1,p_2\}$.
Let $(w_1,w_2, v)$ be weights and set $\sigma_i=w_i^{1-p_i'}$, $i=1,2$.
Then
\[
  \|\mathcal{M}_\alpha^{\mathscr{D}}(f_1\sigma_1, f_2\sigma_2)\|_{L^q(v)}\lesssim [\vec{w},v]_{S_{\vec{P},q}}\prod_{i=1}^2 \|f_i\|_{L^{p_i}(\sigma_i)}.
\]
\end{Lemma}
\begin{proof}
We use similar notations as in Lemma~\ref{lm:l}. Then we have
\begin{eqnarray*}
&&\int_{\bbR^n}\mathcal{M}_\alpha^{\mathscr{D}}(f_1\sigma_1, f_2\sigma_2)^q v dx\\
&\lesssim&\sum_{k,j}\left(\frac{\int_{Q_j^k}|f_1|\sigma_1\int_{Q_j^k}|f_2|\sigma_2}{|Q_j^k|^{2-\alpha/n}} \right)^q v(E(Q_j^k))\\
&=&\sum_{k,j}(E_{Q_j^k}^{\sigma_1}|f_1|)^q\left(\frac{\sigma_1(Q_j^k)\int_{Q_j^k}|f_2|\sigma_2}{|Q_j^k|^{2-\alpha/n}} \right)^q v(E(Q_j^k))\\
&:=& \sum_{k,j}(E_{Q_j^k}^{\sigma_1}|f_1|)^q\eta_j^k,
\end{eqnarray*}
where
\[
   \eta_j^k:=\left(\frac{\sigma_1(Q_j^k)\int_{Q_j^k}|f_2|\sigma_2}{|Q_j^k|^{2-\alpha/n}} \right)^q v(E(Q_j^k)).
\]
Now let $\Omega:=\{(k,j)\}$ and let $\eta$ be the measure on $\Omega$ that assigns mass $\eta_j^k$ to $(k,j)$.
Define
\[
  S: (L^1+L^\infty)(\bbR^n, \sigma_1dx)\rightarrow L^\infty(\Omega, d\eta)
\]
by
\[
  S(g):=\{E_{Q_j^k}^{\sigma_1}|g|\}_{(k,j)\in\Omega}, \quad g\in (L^1+L^\infty)(\bbR^n, \sigma_1dx).
\]
Clearly $S$ is sublinear and of strong-type $(\infty,\infty)$ with norm $1$. Next we show that $S$ is of weak-type $(1,q/{p_1})$.
Let $\lambda>0$ and $\{J_i\}_i$ denotes the maximal cubes relative to the collection
\[
  \{Q_j^k: E_{Q_j^k}^{\sigma_1}|g|>\lambda\}.
\]
We have
\begin{eqnarray*}
\eta\{Sf>\lambda\}&=& \sum_{\{(j,k):\, E_{Q_j^k}^{\sigma_1}|g|>\lambda\}} \eta_j^k =\sum_i \sum_{Q_j^k\subset J_i}\eta_j^k\\
&\le&\sum_i \sum_{Q_j^k\subset J_i}\int_{E(Q_j^k)}\mathcal{M}_\alpha^{\mathscr{D}}(1_{J_i}\sigma_1, f_2 1_{J_i}\sigma_2)^q v dx\\
&\le&\sum_i \int_{J_i}\mathcal{M}_\alpha^{\mathscr{D}}(1_{J_i}\sigma_1, f_2 1_{J_i}\sigma_2)^q v dx\\
&\le&[\vec{w},v]_{S_{\vec{P},q}}^q\sum_i \sigma_1(J_i)^{q/{p_1}}\|f_2 1_{J_i}\|_{L^{p_2}(\sigma_2)}^{q}\qquad \mbox{(Lemma~\ref{lm:l})} \\
&\le&[\vec{w},v]_{S_{\vec{P},q}}^q\left(\sum_i \sigma_1(J_i)^{p/{p_1}}\|f_2 1_{J_i}\|_{L^{p_2}(\sigma_2)}^p\right)^{q/p}\\
&\le&[\vec{w},v]_{S_{\vec{P},q}}^q\left(\sum_i \sigma_1(J_i)\right)^{q/{p_1}}\left(\sum_i \|f_2 1_{J_i}\|_{L^{p_2}(\sigma_2)}^{p_2}\right)^{q/{p_2}}\\
&\le&[\vec{w},v]_{S_{\vec{P},q}}^q\|f_2 \|_{L^{p_2}(\sigma_2)}^{q}\left(\frac{1}{\lambda}\int |g|\sigma_1 dx\right)^{q/{p_1}}.
\end{eqnarray*}
This shows that $S$ is of weak-type $(1,q/{p_1})$ with norm $[\vec{w},v]_{S_{\vec{P},q}}^{p_1}\|f_2 \|_{L^{p_2}(\sigma_2)}^{p_1}$.
Then by Proposition~\ref{prop:inter} again, we get $S$ is of strong-type $(p_1, q)$ with norm $C[\vec{w},v]_{S_{\vec{P},q}}\|f_2 \|_{L^{p_2}(\sigma_2)}$,
which is exactly the following
\[
  \sum_{k,j}\left(\frac{\int_{Q_j^k}|f_1|\sigma_1\int_{Q_j^k}|f_2|\sigma_2}{|Q_j^k|^{2-\alpha/n}} \right)^q v(E(Q_j^k))
  \lesssim [\vec{w},v]_{S_{\vec{P},q}}^q\prod_{i=1}^2\|f_i\|_{L^{p_i}(\sigma_i)}^q.
\]
This completes the proof.
\end{proof}

Now we are ready to prove Theorem~\ref{thm:main}.

\begin{proof}[Proof of Theorem~\ref{thm:main}]
By Proposition~\ref{prop:p1}, we have
\[
  \mathcal{M}_{\alpha}(f_1\sigma_1, f_2\sigma_2)(x)
  \le C_{n}\sum_{t\in\{0,1/3\}^n}\mathcal{M}_{\alpha}^{\mathscr{D}_t}
    (f_1\sigma_1, f_2\sigma_2)(x).
\]
Then by Lemma~\ref{lm:ll}, we have
\begin{eqnarray*}
  \|\mathcal{M}_\alpha(f_1\sigma_1,  f_2\sigma_2)\|_{L^q(v)}\lesssim [\vec{w},v]_{S_{\vec{P}}}\prod_{i=1}^2 \|f_i\|_{L^{p_i}(\sigma_i)}.
\end{eqnarray*}
It follows that
\begin{eqnarray*}
  \|\mathcal{M}_\alpha\|_{L^{p_1}(w_1)\times  L^{p_2}(w_2)\rightarrow L^q(v)}
  &=& \|\mathcal{M}_\alpha(\cdot\sigma_1,\cdot\sigma_2)\|_{L^{p_1}(\sigma_1)\times  L^{p_2}(\sigma_2)\rightarrow L^q(v)}\\
  &\lesssim&[\vec{w},v]_{S_{\vec{P}}}.
\end{eqnarray*}
On the other hand, it is obvious that
\begin{eqnarray*}
  \|\mathcal{M}_\alpha\|_{L^{p_1}(w_1)\times  L^{p_2}(w_2)\rightarrow L^q(v)}
  &=& \|\mathcal{M}_\alpha(\cdot\sigma_1,\cdot\sigma_2)\|_{L^{p_1}(\sigma_1)\times  L^{p_2}(\sigma_2)\rightarrow L^q(v)}\\
  &\ge&[\vec{w},v]_{S_{\vec{P}}}.
\end{eqnarray*}
This completes the proof.
\end{proof}

\section{Further Discussions}

For simplicity, we consider the special case $m=2$ in this section. First, we have the following result.
\begin{Theorem}\label{thm:m1}
Suppose that $0\le \alpha<2n$, that $1<p_1,p_2<\infty$, that $1/{p}=1/{p_1}+1/{p_2}$ and that $1/q=1/p-\alpha/n$.
Let $(w_1,w_2, v)$ be weights and set $\sigma_i=w_i^{1-p_i'}$, $i=1,2$.
Then $\mathcal{M}_\alpha$ is bounded from $L^{p_1}_{w_1}(\bbR^n)\times L^{p_2}_{w_2}(\bbR^n)$ to $L^q_v(\bbR^n)$
if and only if the following conditions hold
\begin{eqnarray*}
&&C_1:=\sup_{Q: \mbox{cubes in $\bbR^n$}}\frac{\left(\int_{Q}\mathcal{M}_\alpha(\sigma_1 1_Q, f_2 1_Q\sigma_2)^q vdx\right)^{1/q}}{\sigma_1(Q)^{1/{p_1}}\|f_2\|_{L^{p_2}(\sigma_2)}}<\infty,\\
&&C_2:=\sup_{Q: \mbox{cubes in $\bbR^n$}}\frac{\left(\int_{Q}\mathcal{M}_\alpha(f_1 1_Q\sigma_1, 1_Q\sigma_2)^q vdx\right)^{1/q}}{\|f_1\|_{L^{p_1}(\sigma_1)}\sigma_2(Q)^{1/{p_2}}}<\infty.
\end{eqnarray*}
\end{Theorem}

\begin{proof}
The necessity is obvious. We only  prove the sufficiency. As in the previous section, it suffices to prove it for the
dyadic fractional maximal operator.

Without loss of generality, we assume that $\|f_i\|_{L^{p_i}(\sigma_i)}=1$, $i=1$, $2$. Note that the general case follows by homogeneity.
Similarly to the proof of Lemma~\ref{lm:ll}, it suffices to estimate the following
\[
  \sum_{Q\in\mathcal{Q}}\left(\prod_{i=1}^2 \frac{1}{|Q|^{1-\alpha/2n}}\int_{Q}|f_i|\sigma_i dy_i\right)^q v(E(Q)).
\]
By the monotone convergence theorem, we can assume that all cubes in $\mathcal{Q}$ are contained in some maximal dyadic cube $\overline{Q}$.
Next, for $i=1,2 $, similarly to \cite{T}, we define
\[
  \mathcal{Q}_i:=\{Q\in\mathcal{Q}: (E_Q^{\sigma_i}|f_i|)^{p_i}\sigma_i(Q)\ge (E_Q^{\sigma_j}|f_j|)^{p_j}\sigma_j(Q)\quad\mbox{for  $j\neq i$}\}.
\]
It is obvious that $\mathcal{Q}=\mathcal{Q}_1\bigcup \mathcal{Q}_2$.
By symmetry, we only need to estimate the following
\begin{eqnarray*}
&&\sum_{Q\in\mathcal{Q}_1}\left(\prod_{i=1}^2 \frac{1}{|Q|^{1-\alpha/2n}}\int_{Q}|f_i|\sigma_i dy_i\right)^q v(E(Q))\\
&=&\sum_{Q\in\mathcal{Q}_1}\left(\prod_{i=1}^2  E_{Q}^{\sigma_i} |f_i|\right)^q\left(\prod_{i=1}^2 \frac{\sigma_i(Q)}{|Q|^{1-\alpha/2n}}\right)^q v(E(Q))\\
&=&\sum_{G\in\mathcal{G}}\sum_{Q\in\mathcal{Q}_1\atop \Gamma(Q)=G}\left(\prod_{i=1}^2  E_{Q}^{\sigma_i} |f_i|\right)^q\left(\prod_{i=1}^2 \frac{\sigma_i(Q)}{|Q|^{1-\alpha/2n}}\right)^q v(E(Q)),
\end{eqnarray*}
where $\mathcal{G}$ is the set of principal cubes with respect to $|f_1|$, $\sigma_1$ and $\mathcal{Q}_1$.
For any $G\in\mathcal{G}$, let $\mathcal{G}^*(G)$ be the collection of maximal cubes $G'\in\mathcal{G}$ such that $G'\subsetneq G$.
Then by the definition of the principal cubes, for any $Q\in\mathcal{Q}_1$ with $\Gamma(Q)=G$ and $G'\in\mathcal{G}^*(G)$, we have
either $G'\subsetneq Q$ or $G'\bigcap Q=\emptyset$.
Denote $U(G)=\bigcup_{G'\in\mathcal{G}^*(G)} G'$.
%, $f_2^0=f_2\cdot 1_{G\setminus U(G)}$ and $f_2^\infty=f_2\cdot 1_{U(G)}$.
Then we have
\begin{eqnarray*}
E_{Q}^{\sigma_2} (|f_2|)&=&E_{Q}^{\sigma_2} (|f_2|\cdot 1_{G\setminus U(G)})+E_{Q}^{\sigma_2} (|f_2|\cdot 1_{U(G)})\\
&=&E_{Q}^{\sigma_2} (|f_2|\cdot 1_{G\setminus U(G)})+E_{Q}^{\sigma_2} (\sum_{G'\in\mathcal{G}^*(G)}(E_{G'}^{\sigma_2}|f_2|)1_{G'}).
\end{eqnarray*}
%For simplicity, we denote
Let
\[
  f_2^0=f_2\cdot 1_{G\setminus U(G)} \quad\mbox{and}\quad f_2^\infty=\sum_{G'\in\mathcal{G}^*(G)}(E_{G'}^{\sigma_2}|f_2|)1_{G'}.
\]
Then we have
\begin{eqnarray*}
&&\sum_{Q\in\mathcal{Q}_1}\left(\prod_{i=1}^2 \frac{1}{|Q|^{1-\alpha/2n}}\int_{Q}|f_i|\sigma_i dy_i\right)^q v(E(Q))\\
&\lesssim& \sum_{l\in\{0,\infty\}}\sum_{G\in\mathcal{G}}(E_G^{\sigma_1}|f_1|)^q\sum_{Q\in\mathcal{Q}_1\atop \Gamma(Q)=G}
\left( E_{Q}^{\sigma_2} |f_2^{l}|\right)^q\cdot\left(\prod_{i=1}^2 \frac{\sigma_i(Q)}{|Q|^{1-\alpha/2n}}\right)^q v(E(Q))\\
&\le&\sum_{l\in\{0,\infty\}}\sum_{G\in\mathcal{G}}(E_G^{\sigma_1}|f_1|)^q\cdot
\int_G \mathcal{M}_\alpha^{\mathscr{D}}(1_G\sigma_1,|f_2^{l}|1_G\sigma_2)^q vdx.
\end{eqnarray*}
By hypothesis, we have
\begin{eqnarray*}
\int_G \mathcal{M}_\alpha^{\mathscr{D}}(1_G\sigma_1,|f_2^{l}|1_G\sigma_2)^q vdx
\lesssim C_1^q \sigma_1(G)^{q/{p_1}}\|f_2^{l}1_G\|_{L^{p_2}(\sigma_2)}^q.
\end{eqnarray*}
It follows that
\begin{eqnarray*}
&&\sum_{G\in\mathcal{G}}(E_G^{\sigma_1}|f_1|)^q
\int_G \mathcal{M}_\alpha(1_G\sigma_1,|f_2^{l}|1_G\sigma_2)^q vdx\\
&\lesssim&C_1^q\sum_{G\in\mathcal{G}}(E_G^{\sigma_1}|f_1|)^q\sigma_1(G)^{q/{p_1}}\|f_2^{l}1_G\|_{L^{p_2}(\sigma_2)}^q\\
&\le&C_1^q\left(\sum_{G\in\mathcal{G}}(E_G^{\sigma_1}|f_1|)^p\sigma_1(G)^{p/{p_1}}
\|f_2^{l}1_G\|_{L^{p_2}(\sigma_2)}^p\right)^{q/p}\\
&\le&C_1^q\left(\sum_{G\in\mathcal{G}}(E_G^{\sigma_1}|f_1|)^{p_1}\sigma_1(G)\right)^{q/{p_1}}
 \left(\sum_{G\in\mathcal{G}}\|f_2^{l}1_G\|_{L^{p_2}(\sigma_2)}^{p_2}\right)^{q/{p_2}}.
\end{eqnarray*}
For $l=0$, we have
\begin{eqnarray*}
\sum_{G\in\mathcal{G}}\|f_2^{l}1_G\|_{L^{p_2}(\sigma_2)}^{p_2}
&=&\sum_{G\in\mathcal{G}}\int_{G\setminus U(G)} |f_2|^{p_2}\sigma_2 dx\\
&\le&\int|f_2|^{p_2}\sigma_2 dx\le 1.
\end{eqnarray*}
For $l=\infty$, we have
\begin{eqnarray*}
\sum_{G\in\mathcal{G}}\|f_2^{l}1_G\|_{L^{p_2}(\sigma_2)}^{p_2}
&=& \sum_{G\in\mathcal{G}}\sum_{G'\in\mathcal{G}^*(G)}(E_{G'}^{\sigma_2}|f_2|)^{p_2}\sigma_2(G')\\
&\le&\sum_{G\in\mathcal{G}}(E_G^{\sigma_1}|f_1|)^{p_1}\sigma_1(G).
\end{eqnarray*}
Then by (\ref{eq:e2}), we have
\[
  \sum_{Q\in\mathcal{Q}_1}\left(\prod_{i=1}^2 \frac{1}{|Q|^{1-\alpha/2n}}\int_{Q}|f_i|\sigma_i dy_i\right)^q v(E(Q))\lesssim C_1^q.
\]
This completes the proof.
\end{proof}

By Theorem~\ref{thm:m1}, we reduce the problem to characterize $C_1$ and $C_2$. By symmetry we concentrate on $C_1$.
Let
\[
  U_Q(f):=\sigma_1(Q)^{-1/{p_1}}\mathcal{M}_\alpha(\sigma_1 1_Q, f 1_Q\sigma_2).
\]
For fixed $Q$, $U_Q$ is a sublinear operator from $L^{p_2}_{\sigma_2}(Q)$ to $L^q_v(Q)$.
It seems difficult to give a characterization for such an operator when $p_2>q$. We do not know whether Sawyer's test condition still applies in this case.


\begin{thebibliography}{EEE}
\bibitem{BL}
J.~Berg and J.~L\"ofstr\"om,
Interpolation Spaces,
Springer-Verlag Berlin, New York, 1976.
\bibitem{Buckley}
S.~Buckley,
Estimates for operator norms on weighted spaces and reverse Jensen inequalities,
Trans. Amer. Math. Soc.,
 340 (1993), 253--272.

\bibitem{CD}
W.~Chen, W.~Dami\'an,
Weighted estimates for the multilinear maximal function,
http://arxiv.org/abs/1304.5999

\bibitem{DLP}
W.~Dami\'{a}n, A.K.~Lerner and C.~P\'erez,
Sharp weighted bounds for multilinear maximal functions and Calder\'{o}n-Zygmund operators,
 http://arxiv.org/abs/1211.5115.

%\bibitem{H}
%T.~Hyt\"{o}nen,
%The sharp weighted bound for general Calder\'{o}n-Zygmund operators,
%Ann. of Math.,
%175 (2012), 1473--1506.

\bibitem{HLMORSU}
T.~Hyt\"{o}nen, M.~Lacey, H.~Martikainen, T.~Orponen, M.~Reguera, E.~Sawyer, I.~Uriarte-Tuero,
Weak and strong type estimates for maximal truncations of Calder\'{o}n-Zygmund operators on $A_p$ weighted spaces ,  	
J. Anal. Math. 118 (2012), 177--220.

\bibitem{HP}
T.~Hyt\"onen and C.~P\'erez,
Sharp weighted bounds involving $A_\infty$,
Analysis \& PDE, (to appear).

%\bibitem{L1}
%M.~Lacey,
%Two weight inequality for the Hilbert transform: A real variable characterization, II,
%http://arxiv.org/abs/1301.4663.

%\bibitem{L2}
%M.~Lacey,
%The two weight inequality for the Hilbert transform: A primer,
%http://arxiv.org/abs/1304.5004.

\bibitem{LMPT}
M.~Lacey, K.~Moen, C.~P\'erez and R.H.~Torres,
Sharp weighted bounds for fractional integral operators,
J. Funct. Anal.,
259 (2010), 1073--1097.

%\bibitem{LSSU}
%M.~Lacey, E.~Sawyer, C.~Shen and I.~Uriarte-Tuero,
%The two weight inequality for Hilbert transform, Coronas, and energy conditions,
%http://arxiv.org/abs/1108.2319.

\bibitem{LSSU1}
M.~Lacey, E.~Sawyer, C.~Shen and I.~Uriarte-Tuero,
Two weight inequality for the Hilbert transform: A real variable characterization, I,
http://arxiv.org/abs/1201.4319.

\bibitem{LSU}
M.~Lacey, E.~Sawyer and I.~Uriarte-Tuero,
A characterization of two weight norm inequalities for maximal singular integrals with one
doubling measure,
J. Anal. \& P.D.E. 5(2012), 1--60.


\bibitem{L}
A.K.~Lerner,
On an estimate of Calder\'on-Zygmund operators by dyadic positive operators,
J. Anal. Math., to appear.

\bibitem{LOPTT}
A.K.~Lerner, S.~Ombrosi, C.~P\'erez, R.H.~Torres and R.~Trujillo-Gonz\'alez,
New maximal functions and multiple weights for the multilinear Calder\'on¨CZygmund theory,
Adv. in Math., 220 (2009), 1222--1264.

\bibitem{LMS}
K.~Li, K.~Moen and W.~Sun,
Sharp weighted inequalities for multilinear fractional maximal operator and fractional integrals,
http://arxiv.org/abs/1304.2973.


\bibitem{MXL}
T.~Mei, Q.~Xue, S.~Lan,
Sharp Weighted Bounds for Multilinear fractional Maximal type Operators with Rough Kernels,
http://arxiv.org/abs/1305.1865


\bibitem{M}
K.~Moen,
Weighted inequalities for multilinear fractional integral operators,
Collect. Math., 60 (2009), 213--238.

\bibitem{M1}
K. Moen,
Sharp one-weight and two-weight bounds for maximal operators,
Studia Mathematica, 194 (2009), 163--180.

\bibitem{Mu}
B.~Muckenhoupt,
Weighted norm inequalities for the Hardy maximal function,
Trans. Amer. Math. J. 19 (1972), 207--226.

\bibitem{MW}
B.~Muckenhoupt and R.~Wheeden,
Weighted norm inequalities for fractional integrals,
Trans. Amer. Math. Soc. 192 (1974) 261--274.

\bibitem{NTV}
F.~Nazarov, S.~Treil, A.~Volberg,
Two weight estimate for the Hilbert transform and corona decomposition for non-doubling measures,
http://arxiv.org/abs/1003.1596.

\bibitem{S}
E.~Sawyer,
A characterization of a two weight norm inequality for maximal operators,
Studia Math., 75 (1982), 1--11.

\bibitem{S1}
E.~Sawyer,
A characterization of two weight norm inequalities for fractional and Poisson integrals,
Trans. Amer. Math. Soc., 308 (1988), 533--545.

\bibitem{SSU}
E.~Sawyer, C.~Shen and I.~Uriarte-Tuero,
A two weight theorem for fractional singular integrals: an expanded version,
http://arxiv.org/abs/1302.5093.

\bibitem{T}
S.~Treil,
A remark on two weight estimates for positive dyadic operators,
http://arxiv.org/abs/1201.1455

\bibitem{W}
J.M. Wilson,
Weighted inequalities for the dyadic square function without dyadic $A_\infty$,
Duke Math. J., 55 (1987), 19--50.


\bibitem{YYZ}
Da.~Yang, Do.~Yang and Y.~Zhou,
Endpoint properties of localized Riesz transforms and fractional integrals associated to Schr\"odinger operators,
Potential Anal, 30 (2009),  271--300.


\end{thebibliography}
\end{document}